\documentclass[10pt,conference,letterpaper,twocolumn]{ieeeconf}
\usepackage{graphicx} 

\IEEEoverridecommandlockouts
\usepackage
{
    graphicx,
    amssymb,
    amsmath,
    dsfont, 
    xcolor,
    mathtools,
    authblk,
    nicefrac,
    array,
    enumerate,
    siunitx,
    tikz,
}

\usetikzlibrary{shapes,arrows}
\usetikzlibrary{calc}

\usepackage[utf8]{inputenc}

\usepackage[pdffitwindow=false,
            plainpages=false,
            pdfpagelabels=true,
            pdfpagemode=UseOutlines,
            pdfpagelayout=SinglePage,
            bookmarks=false,
            colorlinks=true,
            hyperfootnotes=false,
            linkcolor=blue,
            urlcolor=blue!30!black,
            citecolor=green!50!black]{hyperref}

\usepackage{cleveref}

\usepackage{caption}

\graphicspath{{Figures/}}

\newcommand{\R}{\mathbb{R}}                                     
\newcommand{\C}{\mathbb{C}}
\newcommand{\N}{{\mathbb{N}}}                                   
\newcommand{\ddt}{\tfrac{\text{\normalfont d}}{\text{\normalfont d}t}} 

\newcolumntype{C}[1]{>{\centering\let\newline\\\arraybackslash\hspace{0pt}}m{#1}}

\DeclareMathOperator{\diag}{diag}

\newtheorem{assumption}{Assumption}[section]
\newtheorem{theorem}[assumption]{Theorem}

\newtheorem{lemma}[assumption]{Lemma}

\newtheorem{definition}[assumption]{Definition}

\newtheorem{remark}[assumption]{Remark}

\newcommand{\cR}{\mathcal{R}}

\newcommand{\cD}{\mathcal{D}}

\newcommand{\cF}{\mathcal{F}}

\newcommand{\eps}{\varepsilon}

\newcommand{\setdef}[2]{\left\{\, #1 \left|\, \vphantom{#1} #2\right.\right\}}

\DeclareOldFontCommand{\rm}{\normalfont\rmfamily}{\mathrm}

\DeclareMathOperator{\Gl}{\mathbf{Gl}}
\DeclareMathOperator{\rk}{\rm rk}

\newenvironment{smallpmatrix}
{\left(\begin{smallmatrix}}
{\end{smallmatrix}\right)}

{\left[\begin{smallmatrix}}
{\end{smallmatrix}\right]}

\makeatletter
\renewcommand*\env@matrix[1][*\c@MaxMatrixCols c]{%
  \hskip -\arraycolsep
  \let\@ifnextchar\new@ifnextchar
  \array{#1}}

\DeclareRobustCommand{\DRAWLINE}[1]{\tikz{\protect\draw[very thick,#1] (0,-0.5ex)(0,0)--(3.2ex,0);}}
\definecolor{BLUE}{rgb}{0.0,0.0,1.0}%
\definecolor{CYAN}{rgb}{0.0, 01.0, 1.0}%
\makeatother

\title{Asymptotic tracking by funnel control with internal models}
\author{Thomas Berger{$^*$} \and Christoph M. Hackl{$^\dagger$} \and Stephan Trenn{$^\ddag$}
\thanks{Thomas Berger acknowledges funding by the Deutsche Forschungsgemeinschaft (DFG, German Research Foundation) – Project-IDs
362536361 and 471539468.}
\thanks{$^{*}$\texttt{thomas.berger@math.upb.de}, Institut f\"ur Mathematik, Universit\"at Paderborn.}
\thanks{$^{\dagger}$\texttt{christoph.hackl@hm.edu}, Labor für mechatronische und regenerative Energiesysteme (LMRES), Hochschule München.} 
\thanks{$^{\ddag}$\texttt{s.trenn@rug.nl}, Bernoulli Institute, University of Groningen, Netherlands}}

\begin{document}

\maketitle

\thispagestyle{empty}
\pagestyle{empty}
\begin{abstract}
Funnel control achieves output tracking with guaranteed tracking performance for unknown systems and arbitrary reference signals. In particular, the tracking error is guaranteed to satisfy time-varying error bounds for all times (it evolves in the funnel). However, convergence to zero cannot be guaranteed, but the error often stays close to the funnel boundary, inducing a comparatively large feedback gain. This has several disadvantages (e.g.\ poor tracking performance and sensitivity to noise due to the underlying high-gain feedback principle). In this paper, therefore, the usually known reference signal is taken into account during funnel controller design, i.e.\ we propose to combine the well-known internal model principle with funnel control. We focus on linear systems with linear reference internal models and show that under mild adjustments of funnel control, we can achieve asymptotic tracking for a whole class of linear systems (i.e.\ without relying on the knowledge of system parameters).
\end{abstract}

\section{Introduction}
Funnel control was developed in the seminal work~\cite{IlchRyan02b}, see also the survey in~\cite{BergIlch21}. The funnel controller proved to be the appropriate tool for tracking problems in various applications such chemical processes~\cite{IlchTren04}, industrial servo-systems~\cite{Hack17}, underactuated multibody systems~\cite{BergDrue21,BergOtto19}, electrical circuits~\cite{BergReis14a,SenfPaug14}, clinical applications~\cite{PompWeye15}, and autonomous driving~\cite{BergRaue18,BergRaue20}. Funnel control only relies on ``structural system knowledge'' such as (strict) relative degree, bounded-input bounded-output zero dynamics and known sign (or positive definiteness) of the high-frequency gain. Therefore, it is intrinsically robust but also achieves ``tracking with prescribed performance'', i.e. the tracking error evolves within a prescribed region, the so-called performance funnel which is designed by a time-varying funnel boundary. However, the exact error evolution within the funnel is not known; e.g., the error may come arbitrarily close to the funnel boundary, resulting in extraordinary large gains and, therefore, from an implementation point of view, exhibiting massive noise sensitivity.

In this contribution, the problem of asymptotic tracking with concurrent prescribed transient behavior of the tracking error is investigated for linear minimum phase systems by exploiting the internal model principle~\cite{FranWonh75a,Wonh79}. The problem has already been solved for relative degree one systems in~\cite{IlchRyan06a} (using internal models as well) and \cite{LeeTren19}, and, for (nonlinear) systems with arbitrary relative degree, in~\cite{BergIlch21}. However, all three works rely on the assumption of a performance funnel whose width shrinks to zero as time goes to infinity. Hence, all three approaches still exhibit massive noise sensitivity during real-world implementation.  In the context of prescribed performance control (PPC), the asymptotic tracking objective has been tackled by a controller design comprising a locally asymptotically stabilizing controller and an additional PPC module (see e.g.~\cite{KanaRovi20}). However, this approach relies on the existence and availability of the stabilizing controller part, which typically requires some knowledge of the system parameters or the system itself (in the sense of model inversion). Moreover, to the best of our knowledge, so far internal models have not been considered in the context of PPC in general.

In the present paper, we suggest an alternative where we combine funnel control and the internal model principle. Neither do we need such knowledge of a locally stabilizing controller as required for PPC nor do we need a funnel design where the funnel width shrinks to zero as time tends to infinity. Nevertheless, our approach still guarantees asymptotic tracking. To do so, we utilize internal models associated with the reference signal (assumed to be known) and employ only one time-varying gain function depending on a number of design parameters which are chosen sufficiently large such that the objective of asymptotic tracking is achieved and the aforementioned problems are avoided during implementation. In contrast to previous works on funnel control for systems with arbitrary relative degree, as e.g.~in \cite{BergIlch21,BergLe18}, the proposed approach also avoids the involvement of several gain functions.

A related, but different approach to the problem utilizes control Lyapunov barrier functions (CLBFs), see e.g.~\cite{TeeGe09,RomdJaya16}. A drawback of this approach is that either the CLBF candidate is hard to determine and/or requires knowledge of the system parameters, or adaptive laws to approximate the uncertainties in the system parameters must be employed, which severely increase the controller complexity. Compared to this, here we present a simple controller of low complexity, which does not need any knowledge of specific system parameters.


\subsection{System class}\label{Ssec:SysClass}


We consider linear systems of the form
\begin{equation} \label{eq:System-lin}
\begin{aligned}
\dot x(t) &= A x(t) + B u(t), \quad x(0) = x^0 \in \R^n, \\
y(t) & = C x(t),
\end{aligned}
\end{equation}
where  $A \in \R^{n \times n}$, $B, C^\top\in \R^{n\times m}$, with the same number of inputs $u:\R_{\ge 0}\to\R^m$ and outputs $y:\R_{\ge 0}\to\R^m$. We assume that the system has a well-defined and known strict relative degree and it is minimum phase, cf.~\cite{IlchRyan07, Isid95}.
\begin{assumption} \label{Ass:rel_deg}
System~\eqref{eq:System-lin} has strict relative degree $r\in\N$, i.e., $CA^kB = 0$ for all $k=0,\ldots,r-2$, and  $CA^{r-1} B$ is invertible. Furthermore, it is minimum phase, i.e.,
\[
   \forall\, \lambda\in\C\ \text{ with } {\rm Re}\, \lambda \ge 0:\  \rk \begin{bmatrix} A - \lambda I_n & B \\ C & 0\end{bmatrix} = n+m.
\]
\end{assumption}

We introduce the following class of systems.

\begin{definition} \label{Def:SystemClass}
    For $m,r \in \N$,
    a system~\eqref{eq:System-lin} belongs to the class $\Sigma_{m,r}$, if \Cref{Ass:rel_deg} is satisfied and $\Gamma = CA^{r-1}B$ is positive definite. We write $(A,B,C) \in \Sigma_{m,r}$.
\end{definition}

\subsection{Control objective}\label{Ssec:ContrObj}

The objective is to design a dynamic output derivative feedback of the form
\begin{equation}\label{eq:objcontr}
\begin{aligned}
\dot{\xi}(t)\,&=F\big(t,\xi(t), y(t), \dot y(t), \ldots, y^{(r-1)}(t)\big),\quad \xi(0) = \xi^0,\\
u(t)\,&=G\big(t,\xi(t),y(t), \dot y(t), \ldots, y^{(r-1)}(t)\big),
\end{aligned}
\end{equation}
which achieves that, for any reference signal $y_{\rm ref}:\R_{\ge 0}\to\R^m$ within a certain class (defined in \Cref{Ssec:Refs}),
the tracking error $e(t)=y(t)-y_{\rm ref}(t)$ evolves within a prescribed performance funnel
\begin{equation}
\mathcal{F}_{\varphi} := \setdef{(t,e)\in\R_{\ge 0} \times\R^m}{\varphi(t) \|e\| < 1},\label{eq:perf_funnel}
\end{equation}
which is determined by a function~$\varphi$ belonging to
\[
\Phi \!:=\!
\left\{
\varphi\in  C^1(\R_{\ge 0}\to\R)
\left|\!\!\!
\begin{array}{l}
\text{ $\varphi, \dot \varphi$ are bounded,}\\
\text{ $\varphi (t)>0$ for all $t> 0$,}\\
 \text{ and }  \liminf_{t\rightarrow \infty} \varphi(t) > 0
\end{array}
\right.\!\!\!
\right\}.
\]
Furthermore, all signals~$x, u, z$ in the closed-loop system should remain bounded and asymptotic tracking should be achieved, i.e.,
$
    \lim_{t\to\infty} e(t) = 0.
$
The funnel boundary is given by the reciprocal of $\varphi$ as depicted in Fig.~\ref{Fig:funnel}. If $\varphi(0)=0$ , then there is no restriction on the initial value since $\varphi(0) \|e(0)\| < 1$ and the funnel boundary $1/\varphi$ has a pole at $t=0$.

\begin{figure}[h]
\captionsetup[subfloat]{labelformat=empty}
\hspace{1.5cm}
\begin{tikzpicture}[scale=0.35]
\tikzset{>=latex}
  \filldraw[color=gray!15] plot[smooth] coordinates {(0.15,4.7)(0.7,2.9)(4,0.4)(6,1.5)(9.5,0.4)(10,0.333)(10.01,0.331)(10.041,0.3) (10.041,-0.3)(10.01,-0.331)(10,-0.333)(9.5,-0.4)(6,-1.5)(4,-0.4)(0.7,-2.9)(0.15,-4.7)};
  \draw[thick] plot[smooth] coordinates {(0.15,4.7)(0.7,2.9)(4,0.4)(6,1.5)(9.5,0.4)(10,0.333)(10.01,0.331)(10.041,0.3)};
  \draw[thick] plot[smooth] coordinates {(10.041,-0.3)(10.01,-0.331)(10,-0.333)(9.5,-0.4)(6,-1.5)(4,-0.4)(0.7,-2.9)(0.15,-4.7)};
  \draw[thick,fill=lightgray] (0,0) ellipse (0.4 and 5);
  \draw[thick] (0,0) ellipse (0.1 and 0.333);
  \draw[thick,fill=gray!15] (10.041,0) ellipse (0.1 and 0.333);
  \draw[thick] plot[smooth] coordinates {(0,2)(2,1.1)(4,-0.1)(6,-0.7)(9,0.25)(10,0.15)};
  \draw[thick,->] (-2,0)--(12,0) node[right,above]{\normalsize$t$};
  \draw[thick,dashed](0,0.333)--(10,0.333);
  \draw[thick,dashed](0,-0.333)--(10,-0.333);
  \node [black] at (0,2) {\textbullet};
  \draw[->,thick](4,-3)node[right]{\normalsize$\lambda$}--(2.5,-0.4);
  \draw[->,thick](3,3)node[right]{\normalsize$(0,e(0))$}--(0.07,2.07);
  \draw[->,thick](9,3)node[right]{\normalsize$\varphi(t)^{-1}$}--(7,1.4);
\end{tikzpicture}
\caption{Error evolution in a funnel $\mathcal F_{\varphi}$ with boundary $\varphi(t)^{-1}$.}
\label{Fig:funnel}
\end{figure}
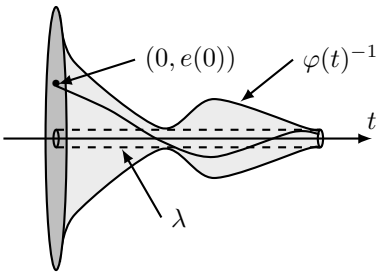

We like to point out that, in contrast to~\cite{BergIlch21,IlchRyan06a}, the objective of asymptotic tracking cannot be achieved by shrinking the width of the funnel to zero as $t\to\infty$; this would mean that $\varphi$ becomes unbounded which is not allowed by definition of the class $\Phi_r$. Furthermore, such an approach would drastically increase noise sensitivity. This problem is also avoided in the present paper.

In fact, boundedness of $\varphi$ implies that there exists $\lambda>0$ such that $1/\varphi(t)\geq\lambda$ for all $t \ge 0$, so each performance funnel $\mathcal{F}_{\varphi}$ is bounded away from zero. The funnel boundary is not necessarily monotonically decreasing, which might be advantageous in applications. In some situations widening the funnel over some later time interval might be beneficial, for instance in the presence of periodic disturbances or strongly varying reference signals.  

\subsection{Class of reference signals}\label{Ssec:Refs}

The reference signals to be tracked are functions $y_{\rm ref}:\R_{\ge 0}\to\R^m$, whose components $y_{i,\rm ref}$, $i=1,\ldots,m$, are solutions of the scalar differential equation $\alpha(\ddt) y_{i,\rm ref} = 0$, where $\alpha(s)\in\R[s]$ is a monic polynomial with the following property:
\begin{equation}\label{eq:prop_alpha}
    \forall\, \lambda\in\C:\ \alpha(\lambda) = 0\ \implies\ \rk \begin{bmatrix} A - \lambda I_n & B \\ C & 0\end{bmatrix} = n+m.
\end{equation}
In other words,  $y_{\rm ref}$ belongs to the class
\[
    \cR(\alpha) := \setdef{w\in C^\infty(\R_{\ge 0},\R^m)}{\alpha(\ddt) w = 0}.
\]
For example, admissible reference signals are constants, ramps, polynomials, sinusoidals and linear combinations thereof~(see, e.g., \cite[Section 7.3.1]{Hack17}). For those cases, $\alpha$ is chosen to have purely imaginary roots, so that condition~\eqref{eq:prop_alpha} is automatically satisfied for minimum-phase systems. The same is true for unbounded exponential reference signals, however, for exponentially decreasing reference signals condition~\eqref{eq:prop_alpha} may be violated and, since the system parameters are not assumed to be known, this situation cannot be detected. On the other hand, the case of exponentially decreasing reference signals is usually not of practical relevance and, furthermore, a random small perturbation of the roots of $\alpha(s)$ makes condition~\eqref{eq:prop_alpha} valid (with probability one).


\section{Internal models}\label{Sec:IntMod}

In a series of seminal works by Francis and Wonham, see e.g.~\cite{FranWonh75a}, the \textit{internal model principle} was developed,  succinctly summarized in~\cite[p.~210]{Wonh79} as

\begin{quote}
    ``every good regulator must incorporate a model of the outside
world  [\ldots \textit{being capable to replicate} \ldots] the dynamic structure of the exogenous signals which the regulator is required
to process''.
\end{quote}

The goal of the internal model is to allow for reduplication of reference signals of class $\cR(\alpha)$. For real-time implementation, a state space realization of the internal model is required. The internal model can be designed as follows:

\textit{Step 1.} For a monic polynomial $\alpha(s)\in\R[s]$ find a Hurwitz polynomial\footnote{A\! polynomial\! $\beta(s)\in\R[s]$\! is\! Hurwitz,\! if\! all\! its\! roots\! have\! negative\! real\! part.} $\beta(s)\in\R[s]$ such that $\alpha(s)$ and $\beta(s)$ are coprime, $\deg \alpha(s) = \deg \beta(s) =: p$ and $\lim_{s\to \infty} \frac{\beta(s)}{\alpha(s)} = 1$.

\textit{Step 2.} Find a minimal realization $(\hat A, \hat b, \hat c, 1)$ of $\frac{\beta(s)}{\alpha(s)}$ with $\hat A\in\R^{p\times p}$ and $\hat b, \hat c^\top \in\R^p$. Then with $\tilde A := \diag(\hat A,\ldots, \hat A)\in\R^{mp\times mp}$, $\tilde B := \diag(\hat b,\ldots,\hat b)\in\R^{mp\times m}$ and $\tilde C := \diag(\hat c,\ldots,\hat c)\in\R^{m\times mp}$ we have that the system
\begin{equation}\label{eq:int_mod}
\begin{aligned}
    \dot z(t)&= \tilde A z(t) + \tilde B w(t),\quad z(0) = z^0\in\R^{mp},\\
    u(t) &= \tilde C z(t) + I_m w(t)
\end{aligned}
\end{equation}
is a minimal realization of $\tfrac{\beta(s)}{\alpha(s)} I_m$, i.e., $(\tilde A, \tilde B,\tilde C, I_m)$ is controllable and observable and
\[
    \tilde C (sI_m -\tilde A)^{-1} \tilde B +I_m = \frac{\beta(s)}{\alpha(s)} I_m.
\]

For more details on the design procedure see also~\cite[Sec.~7.3]{Hack17}. The system~\eqref{eq:int_mod} or, equivalently, $(\tilde A, \tilde B,\tilde C, I_m)$ will be called internal model (of the class $\cR(\alpha)$) in the following. We summarize some important properties of the interconnection of the internal model with the linear system~\eqref{eq:System-lin}, which is illustrated in Fig.~\ref{fig:Illustration of FC+IM+SYS}.

\begin{figure}
\centering
\hspace*{-8mm}
\scalebox{0.9}{%
\tikzstyle{block} = [draw, fill=white, rectangle,
    minimum height=1.5cm, minimum width=3em, align=center,text width = 1.5cm]
\tikzstyle{sum} = [draw, fill=white, circle, node distance=1cm]
\tikzstyle{input} = [coordinate]
\tikzstyle{output} = [coordinate]
\tikzstyle{pinstyle} = [pin edge={to-,thin,black}]
%
\begin{tikzpicture}[auto, node distance=2cm,>=latex']
    \node [input, name=input] {};
    \node [sum, right of=input] (sum) {};
    \node [block, right of=sum,node distance=1.5cm] (controller) { Funnel control\-ler \eqref{eq:fun-con}};
    \node [block, right of=controller, 
            node distance=2.5cm] (internalmodel) {Internal model \eqref{eq:int_mod}};
\node [block, right of=internalmodel, 
            node distance=2.5cm] (system) {Linear system~\eqref{eq:System-lin}};
    \draw [->] (controller) -- node[name=w] {$w$} (internalmodel);
    \draw [->] (internalmodel) -- node[name=u] {$u$} (system);
    \node [output, right of=system] (output) {};
    \coordinate [below of=w] (measurements) {};
    \draw [draw,->] (input) -- node {$-y_{\rm ref}$} (sum);
    \draw [->] (sum) -- node {$e$} (controller);
    \draw [->] (system) -- node [name=y] {$y$}(output);
    \draw [-] (y) |- (measurements);
    \draw [->] (measurements) -| node[pos=0.99] {} (sum);
    \draw [dotted,draw=black] ($(u.south)-(2.3,0.9)$) rectangle ($(u)+(2.3,1.2)$);
    \node [align=right] at ($(u.south)+(0.8,1.1)$) {Augmented system};
\end{tikzpicture}
}
\caption{Illustration of control system with internal model.}
\label{fig:Illustration of FC+IM+SYS}
\end{figure}
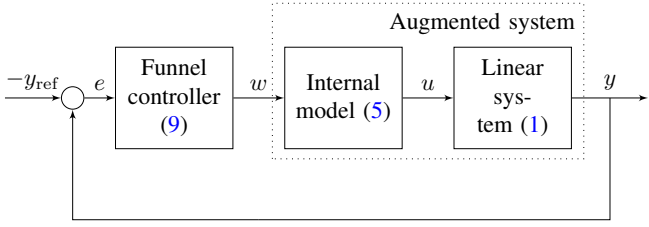

\begin{lemma}\label{Lem:Prop_Intercon}
Consider a system~\eqref{eq:System-lin} with $(A,B,C)\in\Sigma_{m,r}$, let $\alpha(s)\in\R[s]$ be a monic polynomial and $(\tilde A, \tilde B,\tilde C, I_m)$ be an internal model of the class $\cR(\alpha)$. Then the serial interconnection of~\eqref{eq:System-lin} and~\eqref{eq:int_mod}, given by
\begin{equation}\label{eq:Intercon}
\begin{aligned}
    \begin{pmatrix} \dot x(t)\\ \dot z(t)\end{pmatrix}&= A_{\rm ic} \begin{pmatrix} x(t)\\ z(t)\end{pmatrix} + B_{\rm ic} w(t),\\
    y(t) &= C_{\rm ic} \begin{pmatrix} x(t)\\ z(t)\end{pmatrix},\qquad \begin{pmatrix} x(0)\\ z(0)\end{pmatrix} = \begin{pmatrix} x^0\\ z^0\end{pmatrix}
\end{aligned}
\end{equation}
where
\[
   (A_{\rm ic}, B_{\rm ic}, C_{\rm ic}) := \left(\begin{bmatrix} A & B\tilde C\\ 0 & \tilde A\end{bmatrix}, \begin{bmatrix} B\\ \tilde B\end{bmatrix},\begin{bmatrix} C & 0\end{bmatrix}\right),
\]
is in class $\Sigma_{m,r}$ with $C_{\rm ic} A_{\rm ic}^{r-1} B_{\rm ic} = C A^{r-1} B = \Gamma$.
\end{lemma}
\begin{proof}
  The proof is a straightforward extension of that of~\cite[Lem.~7.2]{Hack17} to the multivariable case.
\end{proof}

If~\eqref{eq:Intercon} belongs to $\Sigma_{m,r}$ and hence has strict relative degree $r\in\N$, then by~\cite[Lem.~3.5]{IlchRyan07} there exists a state-space transformation $U\in\Gl_{n+mp}(\R)$ such that $U \begin{smallpmatrix} x(t)\\ z(t)\end{smallpmatrix} = \big( y(t)^\top, \dot y(t)^\top, \ldots, y^{(r-1)}(t)^\top, \eta(t)^\top\big)^\top$, with $\eta:\R_{\ge 0}\to\R^{n+(p-r)m}$, transforms~\eqref{eq:Intercon}  into Byrnes-Isidori form
\begin{equation} \label{eq:BIF}
\begin{aligned}
y^{(r)}(t) & = \sum_{i=1}^{r} R_i y^{(i-1)}(t) + S \eta(t) + \Gamma w(t), \\
\dot \eta(t) & = Q \eta(t) + P y(t),
\end{aligned}
\end{equation}
with initial conditions
\begin{equation}\label{eq:IC}
\begin{aligned}
(y(0),\ldots,y^{(r-1)}(0) ) &= (y_0^0,\ldots,y_{r-1}^0 ) \in \R^{rm}, \\
 \eta(0) &= \eta^0 \in \R^{n+(p-r)m},
\end{aligned}
\end{equation}
where $R_i \in \R^{m \times m}$, $i=1,\ldots,r$, $S,P^\top \in \R^{m \times (n+(p-r)m)}$, $Q \in \R^{(n+(p-r)m) \times (n+(p-r)m)}$ and $\Gamma = CA^{r-1} B$. Furthermore, since~\eqref{eq:Intercon} is minimum phase it follows that $\sigma(Q)\subseteq\C_-$. The second equation in~\eqref{eq:BIF} describes the internal dynamics of~\eqref{eq:Intercon}. If $y=0$ these dynamics are called zero dynamics. For an extensive discussion of the minimum phase property and its relation to the zero dynamics we refer to~\cite{IlchWirt13}.

\section{Controller design}\label{Sec:ContrDes}

In order to achieve the control objective described in \Cref{Ssec:ContrObj} we introduce the following novel funnel controller with additional adaptive gain terms, which is to be applied to the interconnection of~\eqref{eq:System-lin} with an internal model~\eqref{eq:int_mod} of the class $\cR(\alpha)$, where $y_{\rm ref}\in \cR(\alpha)$: 
\begin{equation}\label{eq:fun-con}
\boxed{
\begin{aligned}
    e_1(t) &= e(t) = y(t) - y_{\rm ref}(t),\\
    e_{i+1}(t) &= \dot e_i(t) + k_i e_i(t),\quad i=1,\ldots,r-1,\\
    k(t) &= \frac{k_r}{1-\varphi_r(t)^2 \|e_r(t)\|^2},\\
    w(t)&= -k(t) e_r(t)
\end{aligned}
}
\end{equation}
with the controller design parameters
\begin{equation}\label{eq:con-param}
\boxed{
\begin{aligned}
    k_1,\ldots,k_{r} > 0,\ \varphi_r\in\Phi.
\end{aligned}
}
\end{equation}

Compared to standard funnel control designs~\cite{BergIlch21,BergLe18}, 
the gains $k_1,\ldots,k_{r}$ in the controller~\eqref{eq:fun-con} are selected as constants, which need to be sufficiently large~-- for $k_1,\ldots,k_{r-1}$ this will be made explicit in due course. The gain $k(t)$ is still  time-varying and increases whenever the error $e_r$ is close to the boundary of the performance funnel $\cF_{\varphi_r}$, so that evolution inside the funnel is guaranteed. We will show that this also ensures that the tracking error~$e$ evolves in a prescribed performance funnel, when $\varphi_r$ is chosen accordingly. The feasibility of asymptotic tracking is ensured by the incorporation of the internal model of the reference signal, which renders the tracking problem a stabilization problem for the interconnected system~\eqref{eq:Intercon} with output~$e$.

We like to note that the actual controller consists of the combination of the internal model~\eqref{eq:int_mod} with the funnel controller~\eqref{eq:fun-con}, which is a dynamic output derivative feedback of the form~\eqref{eq:objcontr}. In the sequel we investigate existence of solutions of the initial value problem resulting from the application of the funnel controller~\eqref{eq:fun-con} to the interconnection~\eqref{eq:Intercon}. By a solution of~\eqref{eq:Intercon},~\eqref{eq:fun-con} we mean a function $(x,z):[0,\omega)\to\R^n\times\R^{mp}$, $\omega\in(0,\infty]$,
which is locally absolutely continuous and satisfies $x(0)=x^0$, $z(0)=z^0$, as well as the differential equations in~\eqref{eq:Intercon},~\eqref{eq:fun-con} for almost all $t\in [0,\omega)$.  A solution is called maximal, if it has no right extension that is also a solution.

\section{Funnel control -- main result}\label{Sec:Main}

Before stating the main result, we discuss how the gains $k_1,\ldots,k_{r-1}$ must be chosen so that the evolution of the tracking error in a performance funnel $\cF_{\varphi_1}$ for some $\varphi_1\in\Phi$ is guaranteed.  To this end, choose $\varphi_2,\ldots,\varphi_r\in \Phi$ and $k_1,\ldots,k_{r-1}>0$ such that
\begin{enumerate}
    \item[(K1)] $k_i > \left\|\frac{\dot\varphi_i}{ \varphi_{i}}\right\|_\infty +  \left\|\frac{\varphi_i}{\varphi_{i+1}}\right\|_\infty$ for $i=1,\ldots,r-1$,
    \item[(K2)] $\varphi_i(0) \|e_i(0)\| < 1$ for $i=1,\ldots,r$.
\end{enumerate}
Note that, by definition, $e_i$ depends on $k_1,\ldots,k_{i-1}$, so both (K1) and (K2) are conditions relating the functions $\varphi_i$ and the gains~$k_i$. Invoking the control law~\eqref{eq:fun-con}, we can make the following observation, which is independent from the application to a specific system.

\begin{lemma}\label{lem:epsi}
Let $y,y_{\rm ref}\in C^{r-1}([0,\omega),\R^m)$, $\omega\in(0,\infty]$, and consider the signals defined in the control law~\eqref{eq:fun-con} for design parameters as in~\eqref{eq:con-param}. If $\varphi_r(t) \|e_r(t)\| < 1$ for all $t\in [0,\omega)$, then for all $\varphi_1,\ldots,\varphi_{r-1}\in\Phi$ which satisfy~(K1) and~(K2) we have for $i=1,\ldots,r-1$:
\begin{equation}\label{eq:est-ki}
\forall\,t\ge 0:\  \varphi_i(t) \|e_i(t)\| \leq \eps_i < 1,
\end{equation}
where
\[
  \eps_i :=
\max\left\{\varphi_i(0)\|e_i(0)\|, \tfrac{1}{k_i} \left(\left\|\tfrac{\dot\varphi_i}{ \varphi_{i}}\right\|_\infty +  \left\|\tfrac{\varphi_i}{\varphi_{i+1}}\right\|_\infty\right)\right\}.
\]
In particular, the signals $e_i$ evolve within the performance funnel $\cF_{\varphi_i}$.
\end{lemma}
\begin{proof}
By induction we may assume that $\varphi_{i+1}(t)\|e_{i+1}(t)\|<1$ for all $t\in[0,\omega)$. Then
\begin{align*}
   & \tfrac12 \ddt \varphi_i(t)^2 \|e_i(t)\|^2 = \dot \varphi_i(t) \varphi_i(t) \|e_i(t)\|^2 \\
    &\quad +  \varphi_i(t)^2 e_i(t)^\top \big(e_{i+1} - k_i e_i(t)\big)\\
    &\le \left(\frac{\dot \varphi_i(t)}{\varphi_i(t)} - k_i\right) \varphi_i(t)^2 \|e_i(t)\|^2 + \frac{\varphi_i(t)^2 \|e_i(t)\|}{\varphi_{i+1}(t)}\\
    &\le \left(\left\|\tfrac{\dot\varphi_i}{ \varphi_{i}}\right\|_\infty +  \left\|\tfrac{\varphi_i}{\varphi_{i+1}}\right\|_\infty - k_i \varphi_i(t) \|e_i(t)\|\right)  \varphi_i(t)\|e_i(t)\|.
\end{align*}
for all $t\in[0,\omega)$. Seeking a contradiction, assume that there exists $t_1\in[0,\omega)$ with $\varphi_i(t_1)\|e_i(t_1)\|>\eps_i$. Set $t_0:=\sup\setdef{t\in[0,t_1)}{\varphi_i(t)\|e_i(t)\| = \eps_i}$, which is well defined by $\varphi_i(0)\|e_i(0)\| \le \eps_i$. Then the above estimate implies $\tfrac12 \ddt \varphi_i(t)^2 \|e_i(t)\|^2\le 0$ for all $t\in[t_0,t_1]$, whence
\[
    \eps_i = \varphi_i(t_0)\|e_i(t_0)\| \ge \varphi_i(t_1)\|e_i(t_1)\| > \eps_i,
\]
a contradiction. This completes the proof.
\end{proof}

The above conditions~(K1) and~(K2) will be used as sufficient condition on the controller design parameters as in~\eqref{eq:con-param}. We are now in the position to show that the control~\eqref{eq:fun-con} in conjunction with the  internal model~\eqref{eq:int_mod} achieves the control objective.


\begin{theorem}\label{Thm:FunCon}
Consider a system~\eqref{eq:System-lin} with $(A,B,C)\in\Sigma_{m,r}$, let $\alpha(s)\in\R[s]$ be a monic polynomial satisfying condition~\eqref{eq:prop_alpha} and $(\tilde A, \tilde B,\tilde C, I_m)$ be an internal model of the class $\cR(\alpha)$, resulting in the interconnection~\eqref{eq:Intercon}. Let $x^0\in\R^n$, $z^0\in\R^{mp}$ be initial values, $y_{\rm ref}\in \cR(\alpha)$ be a reference signal, $\varphi_1\in\Phi$ define the desired performance funnel for the tracking error and choose  $\varphi_2,\ldots,\varphi_r\in \Phi$ and $k_1,\ldots,k_{r-1}>0$ such that conditions~(K1) and~(K2) are satisfied. Then the application of the funnel controller~\eqref{eq:fun-con} 
to the interconnection~\eqref{eq:Intercon} yields an
initial-value problem which has a unique maximal solution
 $(x,z):[0,\omega)\to\R^n\times\R^{mp}$, $\omega\in(0,\infty]$,  with the following properties:
	\begin{enumerate}[(i)]
        \item global existence: $\omega = \infty$;
		\item all errors evolve uniformly in their respective performance funnels, that is for all $i=1,\ldots,r$ there exists $\eps_i \in (0,1)$ such that for all $t\ge 0$ we have $\varphi_i(t) \|e_i(t)\| \le \eps_i$;
		\item all signals $x, z, u$ and $k$ in the closed-loop system are bounded;
        \item if $k_r>0$ is sufficiently large, then the tracking error and its first $r-1$ derivatives converge to zero, i.e.,
        \[
            \forall\, i=0,\ldots,r-1:\ \lim_{t\to\infty} e^{(i)}(t) = 0;
        \]
	\end{enumerate}
\end{theorem}
\begin{proof}
\emph{Step 1}: We show existence and uniqueness of a maximal solution of the closed-loop system consisting of the controller~\eqref{eq:fun-con} applied to~\eqref{eq:Intercon}. Define the polynomials
\[
   p_1(s) = 1,\quad  p_i(s) = (s+k_1) \cdots (s+k_{i-1}),\quad i=2,\ldots,r,
\]
and observe that for $e_i$ as in~\eqref{eq:fun-con} we have that $e_i = p_i(\ddt) e$. Further define the relatively open set~$\cD$ in~\eqref{eq:setD}
\begin{figure*}[bt]
\begin{equation}\label{eq:setD}
\cD := \setdef{(t,x,z)\in\R_{\ge 0}\times \R^n\times\R^{mp}}{\varphi_i(t)\|C p_{i}(A) x - p_{i}(\ddt) y_{\rm ref}(t)\| < 1,\ i=1,\ldots,r}
\end{equation}
\end{figure*}
and the function $F:\cD\to \R^n\times\R^{mp}$ by
\begin{align*}
   & F(t,x,z) \\
   & = \begin{pmatrix}
       Ax + B \tilde C z - \frac{k_r B (C p_{r}(A) x - p_{r}(\ddt) y_{\rm ref}(t))}{1-\varphi_r(t)^2 \|C p_{r}(A) x - p_{r}(\ddt) y_{\rm ref}(t)\|^2}\\
       \tilde A z - \frac{k_r \tilde B (C p_{r}(A) x - p_{r}(\ddt) y_{\rm ref}(t))}{1-\varphi_r(t)^2 \|C p_{r}(A) x - p_{r}(\ddt) y_{\rm ref}(t)\|^2}\\
        \|C p_{r}(A) x - p_{r}(\ddt) y_{\rm ref}(t)\|^2
   \end{pmatrix}.
\end{align*}
Then the closed-loop system is equivalent to
\[
   \begin{pmatrix}
       \dot x(t)\\ \dot z(t) 
   \end{pmatrix}
   = F(t,x(t),z(t)),\quad \begin{pmatrix}
       x(0)\\ z(0) 
   \end{pmatrix} = \begin{pmatrix}
       x^0\\ z^0 
   \end{pmatrix}.
\]
Since it follows from Assumption~\ref{Ass:rel_deg} that $y^{(i)}(t) = CA^i x(t)$ for $i=0,\ldots,r-1$, it is clear that
$(0,x^0,z^0)\in\cD$. Furthermore, $F$ is measurable in $t$
and locally Lipschitz in $(x,z)$. Hence, by the theory of
ordinary differential equations (see e.g.~\cite[\S~10, Thm.~XX]{Walt98}) there
exists a unique maximal solution $(x,z):[0,\omega)\to\R^n\times\R^{mp}$, $\omega\in(0,\infty]$, of~\eqref{eq:Intercon},~\eqref{eq:fun-con} satisfying the initial conditions.
Moreover, the closure of the graph of this solution is not a compact subset of~$\cD$. We also note that, by definition of~$\cD$, $\varphi_i(t) \|e_i(t)\| < 1$ for all $t\in [0,\omega)$ and all $i=1,\ldots,r$.

\emph{Step 2}: We show~(ii) for $i=1,\ldots,r-1$ on $[0,\omega)$. Since $\varphi_r(t) \|e_r(t)\| < 1$ for all $t\in[0,\omega)$ was shown in Step 1, this follows directly from Lemma~\ref{lem:epsi}. 

\emph{Step 3}: We derive a differential equation for $e_r$.
By~\cite[Lem.~5.1.2]{Ilch93}\footnote{The result of~\cite[Lem.~5.1.2]{Ilch93} requires $\alpha(s)$ to have only roots with non-negative real parts, however a careful inspection of the proof reveals that condition~\eqref{eq:prop_alpha} suffices.} there exists $v\in C^1(\R_{\ge 0},\R^{n+mp})$ such that, using the notation from Lemma~\ref{Lem:Prop_Intercon},
\[
    \dot v(t) = A_{\rm ic} v(t),\ y_{\rm ref}(t) = C_{\rm ic} v(t).
\]
Set $x_e(t) := \begin{smallpmatrix} x(t)\\ z(t)\end{smallpmatrix} - v(t)$, then
\[
    \dot x_e(t) = A_{\rm ic} x_e(t)+ B_{\rm ic} w(t),\quad e(t) = C_{\rm ic} x_e(t).
\]
By Lemma~\ref{Lem:Prop_Intercon}, $(A_{\rm ic}, B_{\rm ic}, C_{\rm ic}) \in\Sigma_{m,r}$ and hence the above system can be transformed into the form~\eqref{eq:BIF}, i.e., we have
\begin{align*}
 e^{(r)}(t) & = \sum_{i=1}^{r} R_i  e^{(i-1)}(t) + S \eta(t) + \Gamma w(t), \\
\dot \eta(t) & = Q \eta(t) + P  e(t).
\end{align*}
Now let $\mu_1,\ldots,\mu_{r-1}\in\R$ be such that
$
    p_r(s) = s^{r-1} + \sum_{i=1}^{r-1} \mu_{i} s^{i-1}.
$
Then we find
\begin{align*}
    \dot e_r(t) &= \ddt p_r(\ddt) e(t) = e^{(r)}(t) + \sum_{i=1}^{r-1} \mu_{i} e^{(i)}(t)\\
    &= \sum_{i=1}^r R_i  e^{(i-1)}(t) + S\eta(t) + \Gamma w(t) + \sum_{i=1}^{r-1} \mu_{i} e^{(i)}(t).
\end{align*}

\emph{Step 4}: We show (ii) for $i=r$ or, equivalently, that $k$ is bounded on $[0,\omega)$. First observe that, since $y_{\rm ref}$ and $e$ are bounded, and $Q$ is Hurwitz, it follows that $\eta$ is bounded. Furthermore, a straightforward induction utilizing~\eqref{eq:fun-con} gives that
\[
    e^{(i)} = e_{i+1} - \sum_{j=1}^i k_j e_j^{(i-j)} = e_{i+1} + \sum_{j=1}^i c_{i,j} e_j
\]
for some $c_{i,j}\in\R$, $i=1,\ldots,r-1$, $j=1,\ldots,i$. Therefore, since $e_1,\ldots,e_r$ are bounded on $[0,\omega)$, it follows that $e,\dot e, \ldots, e^{(r-1)}$ are bounded on $[0,\omega)$. 
Hence, there exists $C>0$ such that
\begin{align*}
    & \tfrac12 \ddt \|e_r(t)\|^2 \le  - k(t) e_r(t)^\top (\Gamma + \Gamma^\top) e_r(t) + C \|e_r(t)\|\\
    &\le \big(C-k(t) \gamma \|e_r(t)\| \big) \|e_r(t)\|,
\end{align*}
where $\gamma$ is the smallest eigenvalue of the positive definite matrix $\Gamma + \Gamma^\top$. Then, with standard arguments in funnel control as used e.g.\ in~\cite{BergLe18} 
it follows that there exists $\eps_r\in (0,1)$ such that $\varphi_r(t) \|e_r(t)\| \le \eps_r$ for all $t\in[0,\omega)$.

\emph{Step 5}: We show $\omega = \infty$. Seeking a contradiction, assume that $\omega< \infty$. Then, by Steps~2--4, it follows that the graph of the solution $(x,z)$ is a compact subset of~$\cD$, which contradicts the findings of Step~1.

\emph{Step 6}: We show that if $k_r>0$ is large enough, then then $\lim_{t\to\infty} x_e(t) = 0$. Invoking $e^{(i)} = A_{\rm ic}^i x_e$ for $i=0,\ldots,r-1$ it follows that
\[
    e_r = p_r(\ddt) e = C_{\rm ic} p_r(A_{\rm ic}) x_e.
\]
Hence, with $\hat C := C_{\rm ic} p_r(A_{\rm ic})$ we find that the system
\begin{equation}\label{eq:sys-xe}
    \dot x_e(t) = A_{\rm ic} x_e(t) + B_{\rm ic} w(t),\quad e_r(t) = \hat C x_e(t)
\end{equation}
has strict relative degree one as $\hat C B_{\rm ic} = C_{\rm ic} p_r(A_{\rm ic}) B_{\rm ic} = C_{\rm ic} A_{\rm ic}^{r-1} B_{\rm ic} = \Gamma$ by Lemma~\ref{Lem:Prop_Intercon}. Furthermore, the system is minimum phase as
\[
    \det \begin{bmatrix} A_{\rm ic} - \lambda I_n & B_{\rm ic} \\ C_{\rm ic} p_r(A_{\rm ic}) & 0\end{bmatrix} = p_r(\lambda) \det \begin{bmatrix} A_{\rm ic} - \lambda I_n & B_{\rm ic} \\ C_{\rm ic} & 0\end{bmatrix} \neq 0
\]
for all $\lambda\in\C$ with ${\rm Re}\, \lambda \ge 0$, where we have used that $p_r(s)$ is Hurwitz and $(A_{\rm ic}, B_{\rm ic}, C_{\rm ic})$ is minimum phase by Lemma~\ref{Lem:Prop_Intercon}. Then, since $k(t)\ge k_r$, it follows from classical results (see e.g.~\cite[Rem.~2.2.5]{Ilch93}) that there exists $k_r^*>0$ large enough such that for all $k_r\ge k_r^*$ the control $w(t) = - k(t) e_r(t)$ applied to~\eqref{eq:sys-xe} achieves that $\lim_{t\to\infty} x_e(t) = 0$.

\emph{Step 7}: Assertions~(i)--(iii) are shown and it remains to prove~(iv). This follows directly from the observation that $e^{(i)} = A_{\rm ic}^i x_e$ for $i=0,\ldots,r-1$ and Step~6.
\end{proof}

Statement~(iv) of Theorem~\ref{Thm:FunCon} asserts achievement of asymptotic tracking, provided that the parameter $k_r$ is sufficiently large. In order to relax this, future work will concentrate on choosing this parameter adaptively.

\section{Illustrative simulations}
To illustrate the benefits of using internal models in combination with funnel control, comparative simulations have been implemented for the following third-order system:
\begin{equation}\label{eq:example system}
\begin{aligned}
 \dot x(t)&= \begin{bmatrix}  0   &  1  &   0 \\ -3  &   4  &   0 \\ -5  &  0  &  -1 \end{bmatrix} x(t) + \begin{pmatrix} 0\\1\\0\end{pmatrix} u(t), \quad x(0) = \begin{pmatrix} 0\\0\\5\end{pmatrix}\\
    y(t) &= \begin{pmatrix} 1 & 0 & 0\end{pmatrix} x(t).
\end{aligned}
\end{equation}
The system is unstable with eigenvalues $\{-1,3,1\}$, but minimum-phase. Moreover, it has relative degree $r=2$ and positive high-frequency gain $\Gamma = 1$, thus it belongs to $\Sigma_{1,2}$. It is assumed that the instantaneous values $y(t)$ and $\dot{y}(t)$ are available for feedback. For $\omega_0 = 10\pi$, we choose the reference signal $y_{\rm ref}(t) = 2 + \sin(\omega_0 t)$ with derivative $\dot{y}_{\rm ref}(t) = 10\pi\cos(\omega_0 t)$, which is clearly an element of the class $\mathcal{R}(\alpha)$ for $\alpha(s) = s^3 + s \omega_0^2$ with roots $\lambda \in \{0,\pm \jmath \omega_0\}$. Laplace expansion yields
\[
  \forall\, \lambda\in\C:\  \det{\tiny \begin{bmatrix}  -\lambda   &  1  &   0 & 0\\ -3  &   4 - \lambda &   0 & 1\\ -5  &  0  &  -1 - \lambda & 0 \\ 1 & 0 & 0 & 0\end{bmatrix}} = 1 + \lambda,
\]
which shows that \eqref{eq:prop_alpha} is satisfied. Hence, according to Section~\ref{Sec:IntMod}, an appropriate internal model of the form~\eqref{eq:int_mod} is given by (for design details, see \cite[Sec.~7.3.2]{Hack17})
\begin{equation}\label{eq:IM}
\begin{aligned}
    \dot z(t)&= \begin{bmatrix} 0 & 1 & 0 \\ 0 & 0 & 1 \\ 0 & -\omega_0^2 & 0 \end{bmatrix} z(t) + \begin{pmatrix} 0\\0\\1\end{pmatrix} w(t), \quad z(0) = \begin{pmatrix} 0\\0\\0\end{pmatrix}\\
    u(t) &= \begin{pmatrix} 27 & (27-\omega_0^2) & 0\end{pmatrix} z(t) + w(t).
\end{aligned}
\end{equation}

Example system~\eqref{eq:example system} in conjunction with internal model~\eqref{eq:IM} and under funnel control~\eqref{eq:fun-con} has been implemented in Matlab/Simulink (R2023b) using the solver \texttt{ode4 (Runge-Kutta)} with fixed step-size $h=\SI{0.1}{\milli\second}$. The controller tuning parameters were selected as $k_1=74.13$ and $k_2 = 100$ (note that the selections are not trivial as those depend on the funnel boundaries and vice-versa; for details see~\cite{BergDenn23bpp}). Moreover, exponential funnel boundaries were implemented as follows $\psi_1(t) = 1/\varphi_1(t) = (\Lambda_1 - \lambda_1)\exp(-t/T_1) + \lambda_1$ and $\psi_2(t) = 1/\varphi_2(t) = (\Lambda_2 - \lambda_2)\exp(-t/T_2) + \lambda_2$ with $\Lambda_1=10$, $\lambda_1=0.2$, $T_1 = 0.1$\,s and $\Lambda_2=369.76$, $\lambda_2=10.4$, $T_2 = 0.1$\,s, respectively.

Comparative simulation results are plotted in Fig.~\ref{fig:simulation results} for (i) closed-loop system \eqref{eq:fun-con},\eqref{eq:example system}  [\DRAWLINE{CYAN} w/o IM: funnel controller~\eqref{eq:fun-con} is directly applied to example system~\eqref{eq:example system} without internal model, i.e.~$u=w$] and (ii) closed-loop system \eqref{eq:fun-con}, \eqref{eq:IM},\eqref{eq:example system} [\DRAWLINE{BLUE}: funnel controller~\eqref{eq:fun-con} and internal model~\eqref{eq:IM} are applied to example system~\eqref{eq:example system}]. From top to bottom, the plotted time series in the six subplots are: reference $y_{\rm ref}$ and output $y$, boundary $\pm \psi_1:=\pm\tfrac{1}{\varphi_1}$ and error $e = e_1$,  reference derivative $\dot{y}_{\rm ref}$ and output derivative $\dot{y}$, boundary $\pm \psi_2:=\pm\tfrac{1}{\varphi_2}$ and auxiliary error $e_2$, gain $k$, and, finally, controller output $w$ and control action $u$. The time series plots show that with internal model [\DRAWLINE{BLUE}] and without internal model [\DRAWLINE{CYAN} w/o IM], the errors $e_1$ and $e_2$ evolve within their respective funnel regions. However, for the closed-loop system~\eqref{eq:fun-con},\eqref{eq:example system} without internal model [\DRAWLINE{CYAN} w/o IM], both errors do \emph{not} tend to zero but rather oscillate; which in turn leads to significant oscillations in gain $k$ as well as in controller output $w$. In contrast to that, closed-loop system \eqref{eq:fun-con},\eqref{eq:example system},\eqref{eq:IM} with internal model [\DRAWLINE{BLUE}] achieves \emph{asymptotic tracking} with less oscillations in gain $k$ and controller output $w$.

\begin{remark}[Measurement noise]
The use of internal models does not only guarantee asymptotic tracking, but also achieves larger distances to the funnel boundaries. Hence, funnel control with internal models is intrinsically less sensitive to measurement noise. Future research should focus on a rigorous proof of this behavior.
\end{remark}

\begin{figure}
\includegraphics[width=\linewidth]{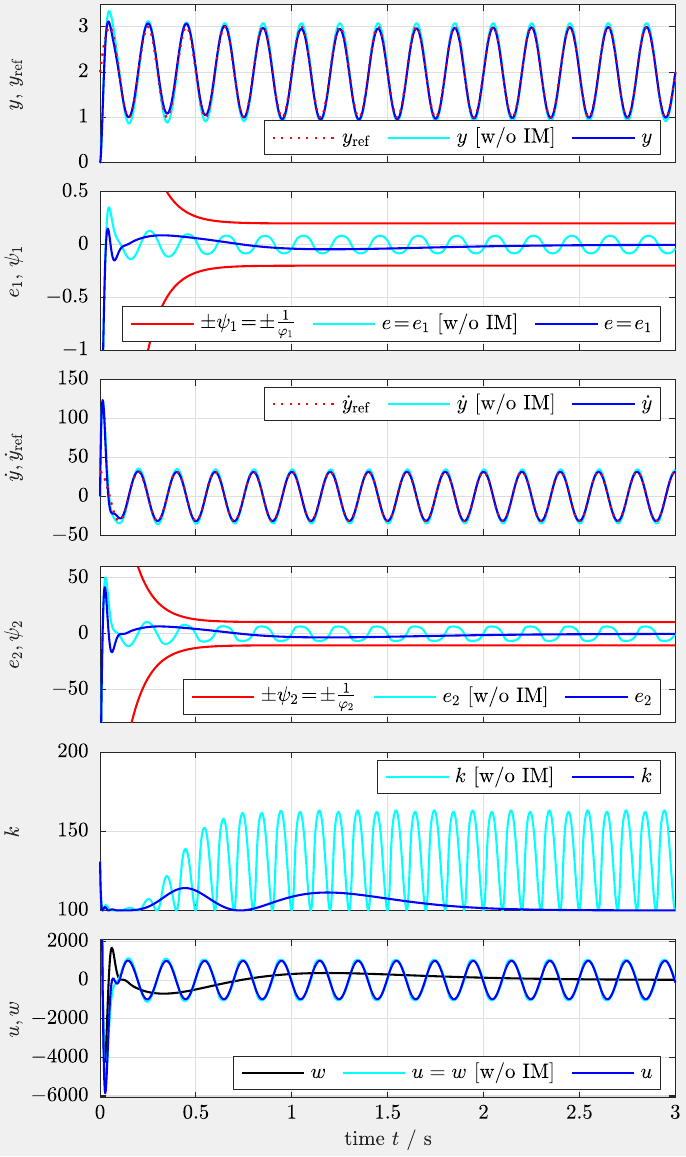}
\caption{Simulation results for system \eqref{eq:example system} using funnel controller \eqref{eq:fun-con} without internal model [\DRAWLINE{CYAN} w/o IM] and funnel controller \eqref{eq:fun-con} with internal model~\eqref{eq:IM} [\DRAWLINE{BLUE}].}
\label{fig:simulation results}
\end{figure}

\bibliographystyle{IEEEtran}
\bibliography{MST,LMRES_Bibliography,local}

\end{document}